\documentclass{scrartcl}

\usepackage{amsmath}
\usepackage{amsfonts}
\usepackage{amssymb}
\usepackage{amsthm}
\usepackage[english]{babel}
\usepackage{graphicx}
\usepackage[format=plain,labelfont=bf,textfont=it]{caption}

\usepackage{enumerate}
\usepackage{mathrsfs}
\usepackage{url}

\usepackage[numbers]{natbib}
\newtheorem{theorem}{Theorem}[section]
\newtheorem{cor}[theorem]{Corollary}
\newtheorem{prop}[theorem]{Proposition}
\theoremstyle{definition}
\newtheorem{defi}[theorem]{Definition}

\newcommand{\cE}{\mathcal{E}}
\newcommand{\cL}{\mathcal{L}}
\newcommand{\cT}{\mathcal{T}}
\renewcommand{\d}{\mathrm{d}}
\newcommand{\D}{\mathrm{D}}
\newcommand{\der}[2]{\frac{\partial #1}{\partial #2}}
\newcommand{\disc}{\mathrm{disc}}
\newcommand{\inner}[2]{\left\langle #1 \,, #2 \right\rangle}
\newcommand{\binner}[2]{\big\langle #1 \,, #2 \big\rangle}
\newcommand{\mesh}{\mathrm{mesh}}
\renewcommand{\mod}{\mathrm{mod}}
\renewcommand{\O}{\mathcal{O}}
\newcommand{\ol}{\overline}
\newcommand{\sk}{\mathrm{skew}}

\newcommand{\hh}{{h \in (0,\infty)}}
\newcommand{\jj}{{j \in \mathbb{Z}}}
\newcommand{\nn}{{n \in \mathbb{N}}}

\DeclareMathOperator{\im}{Im}
\DeclareMathOperator{\re}{Re}

\title{Modified equations for \\ variational integrators applied to \\ Lagrangians linear in velocities}
\author{Mats Vermeeren}
\date{\normalsize Technische Universit\"at Berlin \\
\url{vermeeren@math.tu-berlin.de}}

\begin{document}

\maketitle

\begin{abstract}
\noindent\textbf{Abstract.}
Variational integrators applied to degenerate Lagrangians that are linear in the velocities are two-step methods. The system of modified equations for a two-step method consists of the principal modified equation and one additional equation describing parasitic oscillations. We observe that a Lagrangian for the principal modified equation can be constructed using the same technique as in the case of non-degenerate Lagrangians. Furthermore, we construct the full system of modified equations by doubling the dimension of the discrete system in such a way that the principal modified equation of the extended system coincides with the full system of modified equations of the original system. We show that the extended discrete system is Lagrangian, which leads to a construction of a Lagrangian for the full system of modified equations.
\end{abstract}

\section{Introduction}

An important technique to study the long-time behavior of numerical integrators is backward error analysis. This consists in finding a \emph{modified equation}, a perturbation of the original differential equation whose solutions exactly interpolate the numerical solutions. When a modified equation has been found, one can study the behavior of the numerical solutions by comparing two differential equations, rather than comparing a differential equation with a difference equation. Several long-time (near) conservation laws for symplectic integrators can be proved this way. For a detailed introduction to modified equations we refer to \cite[Chapter IX]{hairer2006geometric}.

In \cite{vermeeren2015modified} we considered modified equations for variational integrators in the case of non-degenerate Lagrangians. We gave a construction for a modified Lagrangian, which produces the modified equation as its Euler-Lagrange equation up to a truncation error of arbitrarily high order. Although the construction was new, the claim that modified equations for variational integrators are Lagrangian was not. This follows by Legendre transformation from the well-known fact that modified equations for symplectic integrators are Hamiltonian. The construction of a modified Lagrangian was combined in \cite{torre2016benefits,torre2017surrogate} with the idea of modifying integrators \citep{chartier2007numerical} to construct variational integrators of improved convergence order.

In this work we extend our previous construction to the case of degenerate Lagrangians that are linear in velocities. In this context the Legendre transformation is not invertible, so the fact that the modified equation is Lagrangian cannot be inferred in the same way from the theory of symplectic integrators. We consider Lagrangians $\cL: T\mathbb{R}^N \cong \mathbb{R}^{2N} \rightarrow \mathbb{R}$ of the form
\begin{equation}\label{lagrangian}
\cL(q, \dot{q}) = \inner{ \alpha(q) }{ \dot{q} } - H(q),
\end{equation}
where $\alpha: \mathbb{R}^N \rightarrow \mathbb{R}^N$, $H: \mathbb{R}^N \rightarrow \mathbb{R}$, and the brackets $\inner{}{}$ denote the standard scalar product. Variational integrators for such Lagrangians were studied for example in \cite{rowley2002variational} and \cite{tyranowski2014variational}. An important role will be played by the matrices
\begin{equation}\label{Askew}
A(q) = \alpha'(q) = \left( \der{\alpha_i(q)}{q_j} \right)_{\!i,j = 1,\ldots,N} \qquad \text{and} \qquad A_\sk(q) = A(q)^T - A(q) .
\end{equation}
We assume that $A_\sk(q)$ is invertible, then the Euler-Lagrange equation for $\cL$ is given by
\begin{equation}\label{cEL}
\dot{q} = A_\sk(q)^{-1} H'(q)^T,
\end{equation}
where $q$ is considered to be a column vector and $H'(q)$ is the row vector of partial derivatives of $H$ with respect to $q_1, \ldots, q_N$. In contrast to the case of non-degenerate Lagrangians, this is a first order ODE.

A well-known example where a Lagrangian of the form \eqref{lagrangian} arises is the dynamics of point vortices in the plane. We will discuss this example in detail in Section \ref{sect-vortex}. Another reason to study this class of Lagrangians is that its extension to PDEs covers several important equations. For example, the nonlinear Schr\"odinger equation is the Euler-Lagrange equation of a Lagrangian whose kinetic term is linear in the time-derivatives (see e.g.\@ \cite[Section 2.1.]{sulem2007nonlinear}). Perhaps the most general application of Lagrangians that are linear in velocities is the variational formulation in phase space of mechanics, where $\cL : T T^* Q \cong \mathbb{R}^{4N} \rightarrow \mathbb{R}$ is given by
\[ \cL(p,q,\dot{p},\dot{q}) = \inner{p}{\dot{q}} - H(p,q). \]
Its Euler-Lagrange equations are Hamilton's canonical equations 
\[ \dot{q} = \left( \der{H}{p} \right)^T \qquad \text{and} \qquad \dot{p} = - \left( \der{H}{q} \right)^T. \] 
Note that even though
$A = \begin{pmatrix}
0 & 0 \\
I & 0
\end{pmatrix}$
is singular in this case, the assumption that $A_\sk$ is invertible still holds. Like many concepts in classical mechanics, the variational principle in phase space dates back to the 19th century \cite[Chapter XXIX]{poincare1899methodes}. A modern treatment can be found for example in \cite[Section 8\---5]{goldstein1980classical}, and an application to geometric integration in \cite{leok2011discrete}.

The construction of modified Lagrangians for variational integrators, which we introduced in \cite{vermeeren2015modified}, carries over to the case of degenerate Lagrangians that are linear in velocities. However, there is a catch. The original differential equation is of first order for the Lagrangians considered here, but the difference equation produced by a variational integrator is of second order. Hence, in this context, variational integrators are two-step methods and parasitic solutions can occur.

In Section \ref{sect-integrators} we present two variational integrators which will be the protagonists of all examples discussed in this work. In  Section \ref{sect-multistep} the essentials of the theory of modified equations for multi-step methods are reviewed. In Section \ref{sect-principal} we summarize the construction of modified Lagrangians from \cite{vermeeren2015modified} and in Section \ref{sect-full} we will present a method to extend it to the full system of modified equations. In Section \ref{sect-ex} we look at some example systems.

\subsubsection*{A note on notation}

As mentioned above, we use the convention that a derivative with respect to a column vector yields a row vector. In particular, this means that the derivative of the scalar product of two column vectors is calculated as
\[ \inner{x}{y}' = \left( x^T y \right)' = x^T y' + y^T x' . \]
Later on we will be taking higher derivatives of vectors with respect to other vectors, resulting in a zoo of tensors. We want to avoid heavy notations using indices, like
\begin{equation}\label{indices}
\sum_{a} A_{a} x^a, \qquad \sum_{a,b} B_{a,b} x^a y^b, \qquad \sum_{a,b,c} C_{a,b,c} x^a y^b z^c, \qquad \cdots.
\end{equation}
If the tensor involved is symmetric, we will use the notations 
\[ A(x), \qquad B(x,y), \qquad C(x,y,z), \qquad \ldots\]
instead. If the tensor is of first or second order, we will often write these expressions as matrix multiplication, 
\[ A x \qquad \text{and} \qquad x^T B y. \]
We will also use the inner product notation $\inner{A^T}{x}$ as an alternative to $Ax$. This allows us to emphasize one particular pairing in a product of more than two tensors. 

Using these notations interchangeably allows us to write equations in an intuitive form and avoid the heavy notation of \eqref{indices}. The downside is that such inconsistent notation could be a source of confusion for the reader. We hope this note is enough to avoid that.

\section{Variational integrators}
\label{sect-integrators}

A variational integrator is a numerical integrator for Lagrangian differential equations, obtained by discretizing the Lagrange function. The action integral $\int \cL(q,\dot{q}) \,\d t$ is replaced by a sum $\sum_j L_\disc(q_j,q_{j+1},h)$. The sequence $(q_j)_\jj$ is a critical point of the action sum if and only if it satisfies the discrete Euler-Lagrange equation
\begin{equation}\label{dEL}
\D_2 L_\disc(q_{j-1},q_j,h) + \D_1 L_\disc(q_j,q_{j+1},h) = 0, 
\end{equation}
where $\D_1$ and $\D_2$ denote partial derivatives with respect to the first and second variable. Assuming this difference equation can be solved for $q_{j+1}$, it provides a numerical approximation of the Euler-Lagrange equations of $\cL$. An excellent overview of the subject of variational integrators is given by Marsden and West \cite{marsden2001discrete}.

For Lagrangians that are linear in the velocities, the continuous Euler-Lagrange equation \eqref{cEL} is of first order, but the discrete Euler-Lagrange equation \eqref{dEL} involves three points, i.e.\@ it is of second order. This means that we are dealing with two-step methods.

We will discuss two examples of variational integrators in detail. Both are obtained by using a simple quadrature rule to approximate the \emph{exact discrete Lagrangian}
\[ L_\mathrm{exact}(q_j,q_{j+1},h) = \int_{jh}^{(j+1)h} \cL(q(t),\dot{q}(t)) \,\d t, \]
where $q(jh) = q_j$, $q((j+1)h) = q_{j+1}$, and $q(t)$ solves the continuous Euler-Lagrange equation.

\subsubsection*{Midpoint rule}

Using $\frac{q_{j+1}-q_j}{2}$ to approximate $\dot{q}$ and the average $\frac{q_j+q_{j+1}}{2}$ to approximate $q$ in the integrand, we find the discrete Lagrangian
\begin{equation}\label{dlag-mp}
L_\disc(q_j,q_{j+1},h) = \inner{ \alpha\!\left( \frac{q_j+q_{j+1}}{2}\right) }{ \frac{q_{j+1} - q_j}{h} } - H\!\left( \frac{q_j+q_{j+1}}{2}\right)
\end{equation}
with discrete Euler-Lagrange equation
\[ \begin{split} 
& \frac{1}{2} \left( \frac{q_{j} - q_{j-1}}{h} \right)^T \alpha'\!\left(\frac{q_{j-1} + q_{j}}{2}\right)  + \frac{1}{2} \left( \frac{q_{j+1} - q_{j}}{h} \right)^T \alpha'\!\left(\frac{q_{j} + q_{j+1}}{2}\right) \\
& - \frac{1}{h} \alpha\!\left(\frac{q_{j} + q_{j+1}}{2}\right)^T + \frac{1}{h} \alpha\!\left(\frac{q_{j-1} + q_{j}}{2}\right)^T - \frac{1}{2} H'\!\left(\frac{q_{j-1} + q_{j}}{2}\right) - \frac{1}{2} H'\!\left(\frac{q_{j} + q_{j+1}}{2}\right) = 0. 
\end{split} \]
In case $\alpha$ is linear, i.e.\@ $\alpha(q)=Aq$, this simplifies to
\[ \frac{q_{j+1} - q_{j-1}}{2h} = A_\sk^{-1}\left(\frac{1}{2} H'\!\left(\frac{q_{j-1} + q_{j}}{2}\right)^T + \frac{1}{2} H'\!\left(\frac{q_{j} + q_{j+1}}{2}\right)^T \right), \]
where $A_\sk$ is defined in Equation \eqref{Askew}. In the case of a non-degenerate Lagrangian this discretization would lead to a variational integrator that is equivalent to the implicit midpoint rule applied to the corresponding symplectic system. Also in the present context we will refer to it as the \emph{midpoint rule}.

\subsubsection*{Trapezoidal rule}

To obtain the second discretization we use the trapezoidal quadrature rule to approximate the exact discrete Lagrangian: we take the average of the integrand evaluated with $q = q_j$ and with $q = q_{j+1}$, while still using $\frac{q_{j+1}-q_j}{2}$ to approximate the derivative $\dot{q}$. We find the discrete Lagrangian
\begin{equation}\label{dlag-t}
L_\disc(q_j,q_{j+1},h) = \inner{ \frac{1}{2} \alpha(q_j) + \frac{1}{2} \alpha(q_{j+1}) }{ \frac{q_{j+1} - q_j}{h} } - \frac{1}{2} H(q_j) - \frac{1}{2} H(q_{j+1})
\end{equation}
with discrete Euler-Lagrange equation
\[  \left( \frac{q_{j+1} - q_{j-1}}{2h} \right)^T \alpha'(q_j) - \frac{\alpha(q_{j+1})^T - \alpha(q_{j-1})^T}{2h} - H'(q_j) = 0. \]
In case $\alpha$ is linear this simplifies to
\[ \frac{q_{j+1} - q_{j-1}}{2h} = A_\sk^{-1}H'(q_j)^T. \]
This discretization is sometimes called the explicit midpoint rule, but we will not use this name to avoid confusion with the previous method. Instead we call this method the \emph{trapezoidal rule}. In the case of a non-degenerate Lagrangian the trapezoidal rule would lead to the St\"ormer-Verlet method.

\section{Modified equations for multistep methods}
\label{sect-multistep}

The classical theory of modified equations does not capture parasitic solutions of multistep methods. An extension of this theory for linear multistep methods was developed by Hairer \cite{hairer1999backward}. (See also \cite[Chapter XV]{hairer2006geometric}.) Here we mention some of the main results, restricted to the case of two-step methods.

For a first order ODE $\dot{q} = f(q)$, consider the linear two-step method
\begin{equation}\label{twostep}
\frac{ a_0 q_j + a_1 q_{j+1} + a_2 q_{j+2} }{h} = b_0 f(q_j) + b_1 f(q_{j+1}) + b_2 f(q_{j+2}) .
\end{equation}
We call the method \eqref{twostep} \emph{symmetric} if $a_0 = -a_2$, $a_1 = 0$, and $b_0 = b_2$. We say that it is \emph{stable} if all roots of the polynomial $\rho(\zeta) = a_0 + a_1 \zeta + a_2 \zeta^2$ satisfy $|\zeta| \leq 1$, and the roots with $|\zeta| = 1$ are simple. A method is stable if and only if the numerical solution for $\dot{q} = 0$ is bounded for any initial condition. Note that the trapezoidal rule is a stable symmetric linear two-step method, but that the midpoint rule is not of the form \eqref{twostep}.

The theory of modified equations for one-step methods is easily extended to yield the following. 

\begin{prop}[{Special case of \cite[Theorem XV.3.1]{hairer2006geometric}}]\label{prop-principal}
Consider a consistent method of the form \eqref{twostep}. Then there exist unique functions $(f_n(q))_\nn$ such that for every truncation index $k$, every solution of
\begin{equation}\label{principal-modeqn-tr}
\dot{q} = f(q) + h f_1(q) + h^2 f_2(q) + \ldots + h^k f_k(q)
\end{equation}
satisfies
\begin{align*}
\frac{ a_0 q(t) + a_1 q(t+h) + a_2 q(t+2h) }{h} = b_0 f(q(t)) + b_1 f(q(t+h)) + b_2 f(q(t & +2h)) \\
&+ \O(h^{k+1}). 
\end{align*}
\end{prop}

In general the right hand side of Equation \eqref{principal-modeqn-tr} will not converge as $k \rightarrow \infty$. Nevertheless, we will call the formal differential equation
\begin{equation}\label{principal-modeqn}
\dot{q} = f(q) + h f_1(q) + h^2 f_2(q) + \ldots 
\end{equation}
the \emph{principal modified equation}. Up to truncation errors, every solution of the principal modified equation gives a solution of the difference equation when evaluated on a mesh $t_0 + h \mathbb{Z}$. However, not every solution of the difference equation can be obtained this way. The solutions that are missed are exactly the parasitic solutions.

\begin{prop}[{Special case of Theorem XV.3.5 from \citep{hairer2006geometric}}]\label{prop-full}
Assume that the method \eqref{twostep} is stable, consistent, and symmetric. Then there exist functions $(f_n(x,y))_{n \in \mathbb{N}}$ and $(g_n(x,y))_{n \in \mathbb{N}}$ such that for every truncation index $k$, for every solution of
\begin{align}
\dot{x} &= f_0(x,y) + h f_1(x,y) + \ldots + h^k f_k(x,y) \label{trunc-principal}\\
\dot{y} &= g_0(x,y) + h g_1(x,y) + \ldots + h^k g_k(x,y), \label{trunc-parasitic}
\end{align}
with $y(0) = \O(h)$, the discrete curve $q_j = x(t+jh) + (-1)^j y(t+jh)$ satisfies
\[ \frac{ a_0 q_j + a_1 q_{j+1} + a_2 q_{j+2} }{h} = b_0 f(q_j) + b_1 f(q_{j+1}) + b_2 f(q_{j+2}) + \O(h^{k+1} ) \]
for every choice of $t$.
\end{prop}

We will call the corresponding system of formal differential equations
\begin{align}
\dot{x} &= f_0(x,y) + h f_1(x,y) + h^2 f_2(x,y) + \ldots , \label{modsys-principal}\\
\dot{y} &= g_0(x,y) + h g_1(x,y) + h^2 g_2(x,y) + \ldots , \label{modsys-parasitic}
\end{align}
the \emph{full system of modified equations}. We call Equation \eqref{modsys-parasitic} the \emph{parasitic modified equation}.

If $y=0$, then Equation \eqref{modsys-principal} reduces to the principal modified equation \eqref{principal-modeqn} and Equation \eqref{modsys-parasitic} reads $\dot{y} = 0$. Hence to determine whether parasitic solutions become dominant over time we need to determine the stability of the invariant manifold $\{y=0\}$ of the system \eqref{modsys-principal}\---\eqref{modsys-parasitic}.

In general, even if the difference equation is not of the form \eqref{twostep}, we have the following definition.

\begin{defi}\label{defi}
Let $\Phi(q_{j-1},q_j,q_{j+1},h)$ be a consistent discretization of some function $F(q,\dot{q})$. 

\begin{enumerate}[$(a)$]
\item Equation \eqref{principal-modeqn} is the \emph{principal modified equation} for the difference equation
\begin{equation}\label{eq1-def}
 \Phi(q_{j-1},q_j,q_{j+1},h) = 0
\end{equation}
if for every truncation index $k$, every solution of the truncated equation \eqref{principal-modeqn-tr} satisfies 
\[ \Phi(q(t-h),q(t),q(t+h),h) = \O(h^{k+1}) \]
at all times $t$.

\item The system of equations \eqref{modsys-principal}\---\eqref{modsys-parasitic} is the \emph{full system of modified equations} for the Equation \eqref{eq1-def} if for every truncation index $k$, for every solution $(x,y)$ of the truncated system \eqref{trunc-principal}\---\eqref{trunc-parasitic}, the discrete curve $q_j = x(t+jh) + (-1)^j y(t+jh)$ satisfies
\[  \Phi(q_{j-1},q_j,q_{j+1},h) = \O(h^{k+1}) \]
for all choices of $t$.
\end{enumerate}
\end{defi}

\section{A Lagrangian for the principal modified equation}
\label{sect-principal}

In \cite{vermeeren2015modified} we constructed a modified Lagrangian for variational integrators in the case of non-degenerate Lagrangian systems. A straightforward adaptation of this construction will give us a Lagrangian for the principal modified equation. Here we present the construction and a rough sketch of the proof. The details of the proof are perfectly analogous to the non-degenerate case, so we refer to \cite{vermeeren2015modified} for their discussion.


We identify points $q_j$ of a numerical solution with step size $h$ with evaluations $q(jh)$ of an interpolating curve. Using a Taylor expansion we can write the discrete Lagrangian $L_\disc(q_{j-1},q_j,h)$ as a function of the interpolating curve $q$ and its derivatives, all evaluated at the point $jh - \frac{h}{2}$,
\[ \cL_\disc([q] ,h)
:= L_\disc\!\left( q - \frac{h}{2}\dot{q} + \frac{1}{2} \left(\frac{h}{2}\right)^2 \ddot{q} - \ldots ,\  q + \frac{h}{2}\dot{q} + \frac{1}{2} \left(\frac{h}{2}\right)^2 \ddot{q} + \ldots, h \right), \]
where the square brackets denote dependence on $q$ and any number of its derivatives.

We want to write the discrete action
\[ S_\disc((q_j)_\jj,h) = \sum_{j=1}^n h L_\disc(q_{j-1},q_j,h) = \sum_{j=1}^n h \cL_\disc\!\left( \left[q \left( jh - \tfrac{h}{2}\right)\right],h \right) \] 
as an integral. This can be done using the Euler-Maclaurin formula. We obtain the \emph{meshed modified Lagrangian}
\begin{align*}
\cL_\mesh([q(t)],h)
:\!&= \sum_{i=0}^\infty \left(2^{1-2i}-1\right) \frac{h^{2i} B_{2i}}{(2i)!} \frac{\d^{2i}}{\d t^{2i}} \cL_\disc([q(t)],h) \\
&= \cL_\disc([q(t)],h) - \frac{h^2}{24} \frac{\d^2}{\d t^2} \cL_\disc([q(t)],h) + \frac{7h^4}{5760} \frac{\d^4}{\d t^4} \cL_\disc([q(t)],h) + \ldots,
\end{align*}
where $B_{2i}$ are the Bernoulli numbers. The power series defining $\cL_\mesh$ generally does not converge. Formally, it satisfies 
\[ S_\disc((q(jh))_\jj,h) = \int \cL_\mesh([q(t)],h) \,\d t . \] 
Note that $\cL_\mesh$ depends on higher derivatives of $q$. Below we will construct a modified Lagrangian that only depends on $q$ and $\dot{q}$. 

The word \emph{meshed} refers to the fact that the discrete system provides additional structure for the continuous variational problem. In the \emph{meshed variational problem}, non-differentiable curves are admissible as long as their singular points are consistent with the mesh, i.e.\@ if they occur at times that are an integer multiple of $h$ away from each other. This imposes additional conditions on critical curves, related to the natural boundary conditions and to the Weierstrass-Erdmann Corner conditions (see e.g.\@ \cite[Sec.\@ 6 and 13]{gelfand1963calculus} for these concepts). These conditions are
\begin{equation}\label{nic}
\forall \ell \geq 2: \quad \frac{\partial \cL}{\partial q^{(\ell)}}(t) = 0.
\end{equation}
We will call them the \emph{natural interior conditions}. Because the action integral of $\cL_\mesh$ equals the discrete action, variations supported on a single mesh interval (i.e.\@ in between consecutive points of the discrete curve) do not change the action integral of $\cL_\mesh$. This implies that the natural interior conditions are automatically satisfied on solutions of the Euler-Lagrange equation (for the particular Lagrangian $\cL_\mesh$, but not in general).

Consider the Euler-Lagrange equation of $\cL_\mesh$,
\[ \sum_{j=0}^\infty (-1)^j \frac{\d^j}{\d t^j} \der{\cL_\mesh}{q^{(j)}} = 0. \]
Because the natural interior conditions \eqref{nic} are automatically satisfied on critical curves, it is equivalent to
\[ \der{\cL_\mesh}{q} - \frac{\d}{\d t} \der{\cL_\mesh}{\dot{q}} = 0 . \]
This equation is of the form
\[ \cE_0(q,\dot{q}) + h \cE_1(q,\dot{q},\ddot{q}) + h^2 \cE_2 \big( q,\dot{q},\ddot{q},q^{(3)} \big) + \ldots = 0 . \]
In the leading order we find a first order differential equation, which we can use to eliminate higher derivatives in the next order (assuming that the derivatives of $q$ are bounded as $h \rightarrow 0$). This can be applied recursively up to any order. Hence we can write the Euler-Lagrange equation formally as a first order differential equation, say
\begin{equation}\label{ih1}
\dot{q} = F(q,h).
\end{equation}
Then expressions for all higher derivatives follow by differentiation and substitution,
\begin{equation}\label{ih2}
\ddot{q} = F_2(q,h) , \qquad q^{(3)} = F_3(q,h), \qquad \ldots .
\end{equation}
The assumption that the derivatives of $q$ are bounded as $h \rightarrow 0$ is not restrictive in practice. The same assumption is necessary to state many other results regarding modified equations rigorously. Families of curves $(q_h)_\hh$ that satisfy this condition are called \emph{admissible families} in \cite{vermeeren2015modified}. In particular there holds for admissible families that if the functions are small, $q_h = \O(h^k)$, then so are their derivatives, $q_h^{(\ell)} = \O(h^k)$. We will use this implicitly later on.

Using \eqref{ih1} and \eqref{ih2} we can replace second and higher derivatives in the meshed Lagrangian to find a first order \emph{modified Lagrangian}, 
\[ \cL_{\mod}(q,\dot{q},h) =  \cL_\mesh([q],h) \,\Big|_{q^{(j)} = F_j(q,h),\ \forall j \geq 2} . \]
Or, avoiding formal power series, a truncated modified Lagrangian
\[ \cL_{\mod,k}(q,\dot{q},h) = \cT_{k}\!\left( \cL_\mesh([q],h) \,\Big|_{q^{(j)} = F_j(q,h),\ \forall j \geq 2} \right), \]
where $\cT_{k}$ denotes truncation of the power series after order $k$. In general the replacements $q^{(j)} = F_j(q,h)$ would change the Euler-Lagrange equations, but because of the natural interior conditions \eqref{nic} this is not the case here. Indeed, one finds
\[
\der{\cL_{\mod,k}}{q} = \cT_{k}\!\left( \der{\cL_\mesh}{q} + \sum_{\ell = 2}^\infty \der{\cL_\mesh}{q^{(\ell)}}\der{F_\ell(q,h)}{q} \right) = \cT_{k}\!\left( \der{\cL_\mesh}{q} \right)
\] 
and
\[
\der{\cL_{\mod,k}}{\dot q} = 
\cT_{k}\!\left( \der{\cL_\mesh}{\dot q}\right) .
\]
It follows that
\[ \der{\cL_{\mod,k}}{q} - \frac{\d}{\d t} \der{\cL_{\mod,k}}{\dot q} = 
\cT_{k}\!\left( \sum_{j=0}^\infty (-1)^j \frac{\d^j}{\d t^j} \der{\cL_\mesh}{q^{(j)}} \right), \]
so up to a truncation error, both Lagrangians yield the same Euler-Lagrange equations. Note that the natural interior conditions \emph{do not} imply that $\partial \cL_\mesh / \partial\dot{q} = 0$, so replacing first derivatives using $\dot{q} = F(q,h)$ is not allowed!
%

The details presented in \cite{vermeeren2015modified} carry over to the degenerate case and yield the following result.
\begin{theorem}
Consider a discrete Lagrangian that is a consistent discretization of a Lagrangian of the form \eqref{lagrangian}. Let $\cL$ be either $\cL_\mesh$ or $\cL_{\mod,k}$, derived from this discrete Lagrangian. Solve the equation 
\[ \der{\cL}{q} - \frac{\d}{\d t} \der{\cL}{\dot{q}} = 0 \]
for $\dot{q}$, and truncate the resulting power series after order $k$. The result,
\[ \dot{q} = f(q) + h f_1(q) + h^2 f_2(q) + \ldots + h^k f_k(q), \]
is a truncation of the principal modified equation.
\end{theorem}

\subsubsection*{Midpoint rule}
From the discrete Lagrangian \eqref{dlag-mp} we find
\begin{align*}
\cL_\disc([q],h) &= L_\disc \!\left( q - \frac{h}{2}\dot{q} + \frac{h^2}{8}\ddot{q} - \ldots
\, , \, q + \frac{h}{2}\dot{q} + \frac{h^2}{8}\ddot{q} + \ldots, h \right) \\
&= \inner{ \alpha \!\left( q + \frac{h^2}{8}\ddot{q} + \ldots \right)\! }{ \dot{q} + \frac{h^2}{24} q^{(3)} + \ldots } - H \!\left( q + \frac{h^2}{8}\ddot{q} + \ldots \right) \\
&= \inner{ \alpha(q) }{ \dot{q} } - H(q)  + \frac{h^2}{24} \left( \binner{ \alpha(q) }{ q^{(3)} } + 3 \inner{ \alpha'(q)\ddot{q} }{ \dot{q} } - 3 H'(q)\ddot{q} \right) + \O(h^4).
\end{align*}
It follows that
\begin{align*}
\cL_\mesh([q],h) &= \inner{ \alpha }{ \dot{q} } - H \\
&\qquad + \frac{h^2}{24} \left( 2 \inner{ A_\sk \dot{q} }{ \ddot{q} } - \inner{ \alpha''(\dot{q},\dot{q}) }{ \dot{q} } - 2 H'\ddot{q} + H''(\dot{q},\dot{q}) \right)  + \O(h^4),
\end{align*}
where the argument $q$ of $A_\sk$, $\alpha$, $H$, and their derivatives is omitted. From this expression we obtain $\cL_{\mod,3}$ by replacing all second derivatives of $q$ using the derivative of the leading order equation, 
\[\ddot{q} = \frac{\d}{\d t}\left( A_\sk(q)^{-1} H'(q)^T \right) + \O(h^2). \]
In case that $\alpha$ is linear we have
\begin{equation}
\ddot{q} = A_\sk^{-1} H'' \dot{q} + \O(h^2) \label{ddq1}
\end{equation}
and we find the following expression for the modified Lagrangian (truncated after $h^3$):
\[ \cL_{\mod,3} = \dot{q}^T A q - H + \frac{h^2}{24} \left( -\dot{q}^T H''\dot{q}  - 2 H' A_\sk^{-1}H''\dot{q} \right).
\]

\subsubsection*{Trapezoidal rule}
From the discrete Lagrangian \eqref{dlag-t} we find
\begin{align*}
\cL_\disc([q],h) &= \inner{ \frac{1}{2}\alpha \!\left( q - \frac{h}{2}\dot{q} + \frac{h^2}{8}\ddot{q} \right)\! + \frac{1}{2}\alpha \!\left( q + \frac{h}{2}\dot{q} + \frac{h^2}{8}\ddot{q} \right)\! }{ \dot{q} + \frac{h^2}{24} q^{(3)} } \\
&\quad - \frac{1}{2} H \!\left( q - \frac{h}{2}\dot{q} + \frac{h^2}{8}\ddot{q} \right) - \frac{1}{2} H \!\left( q + \frac{h}{2}\dot{q} + \frac{h^2}{8}\ddot{q} \right) + \O(h^4) \\
&= \inner{ \alpha }{ \dot{q} } - H \\
&\quad + \frac{h^2}{8} \left( \frac{1}{3} \binner{ \alpha }{ q^{(3)} } +  \inner{ \alpha'\ddot{q} }{ \dot{q} } + \inner{ \alpha''(\dot{q},\dot{q}) }{ \dot{q} } - H'\ddot{q} - H''(\dot{q},\dot{q}) \right) + \O(h^4)
\end{align*}
and
\begin{align*}
\cL_\mesh([q],h) &= \inner{ \alpha }{ \dot{q} } - H + \frac{h^2}{12} \left( \inner{ A_\sk \dot{q} }{ \ddot{q} } + \inner{ \alpha''(\dot{q},\dot{q}) }{ \dot{q} } - H'\ddot{q} - H''(\dot{q},\dot{q}) \right) + \O(h^4).
\end{align*}
Again we assume that $\alpha$ is linear. Using Equation \eqref{ddq1} we find the modified Lagrangian
\[
 \cL_{\mod,3} = \dot{q}^T A q - H + \frac{h^2}{12} \left( - 2 \dot{q}^T H''\dot{q}  - H' A_\sk^{-1}H''\dot{q} \right).
\]

\section{The full system of modified equations}\label{sect-full}

For linear symmetric two-step methods, Proposition \ref{prop-full} describes the full system of modified equations. Here we will show that for variational integrators, without assuming linearity, the full system of modified equations is of the same form. In order to construct the system of modified equations, we split the variable $q_j$ of the discrete system into two parts,
\[ q_j = x_j +  (-1)^j y_j. \]
The motivation for this is that we want to use one variable, $x_j$, to encode the principal behavior and the other, $y_j$, for the parasitic behavior. This is inspired by the formula $q_j = x(t+jh) + (-1)^j y(t+jh)$ from Proposition \ref{prop-full} and Definition \ref{defi}.

\subsection{The Lagrangian approach}

A key property of the doubling of variables is that the extended system is still variational.

\begin{prop}\label{prop-extended}
The discrete curve $(x_j,y_j)_\jj$ is critical for 
\[
\widehat{L}(x_j,y_j,x_{j+1},y_{j+1},h) = \frac{1}{2} L(x_j + y_j, x_{j+1} - y_{j+1},h) + \frac{1}{2} L(x_j - y_j, x_{j+1} + y_{j+1},h),
\]
if and only if the discrete curves $(q_j^+)_\jj$ and $(q_j^-)_\jj$, defined by $q_j^\pm = x_j \pm (-1)^j y_j$, are critical for $L(q_j,q_{j+1},h)$.
\end{prop}

\begin{proof}
The discrete Euler-Lagrange equations for $\widehat{L}(x_j,y_j,x_{j+1},y_{j+1},h)$ are
\begin{align*}
& \frac{1}{2}\D_2 L(x_{j-1}+y_{j-1},x_j-y_j,h) + \frac{1}{2}\D_2 L(x_{j-1}-y_{j-1},x_j+y_j,h) \\
+ & \frac{1}{2}\D_1 L(x_j+y_j,x_{j+1}-y_{j+1},h) + \frac{1}{2}\D_1 L(x_j-y_j,x_{j+1}+y_{j+1},h) = 0
\end{align*}
and
\begin{align*}
- & \frac{1}{2}\D_2 L(x_{j-1}+y_{j-1},x_j-y_j,h) + \frac{1}{2}\D_2 L(x_{j-1}-y_{j-1},x_j+y_j,h) \\
+ & \frac{1}{2}\D_1 L(x_j+y_j,x_{j+1}-y_{j+1},h) - \frac{1}{2}\D_1 L(x_j-y_j,x_{j+1}+y_{j+1},h) = 0 .
\end{align*}
Taking the sum resp.\@ the difference of these equations we find
\begin{align*}
& \D_2 L(x_{j-1}-y_{j-1},x_j+y_j,h) + \D_1 L(x_j+y_j,x_{j+1}-y_{j+1},h) = 0, \\
& \D_2 L(x_{j-1}+y_{j-1},x_j-y_j,h) + \D_1 L(x_j-y_j,x_{j+1}+y_{j+1},h) = 0.
\end{align*}
Depending on the parity of $j$, either the first or the second of those equations is
\[ \D_2 L(q_{j-1}^+,q_j^+,h) + \D_1 L(q_j^+,q_{j+1}^+,h) = 0. \]
The other one is
\[ \D_2 L(q_{j-1}^-,q_j^-,h) + \D_1 L(q_j^-,q_{j+1}^-,h) = 0. \]
Hence $(x_j,y_j)_\jj$ satisfies the Euler-Lagrange equations for $\widehat{L}(x_j,y_j,x_{j+1},y_{j+1},h)$ if and only if $(q_j^+)_\jj$ and $(q_j^-)_\jj$ satisfy the Euler-Lagrange equation for $L(q_j,q_{j+1},h)$.
\end{proof}

\begin{theorem}\label{thm}
Let
\begin{equation}\label{fullmod}
\begin{split}
\dot{x} &= f_0(x,y) + h f_1(x,y) + \ldots + h^k f_k(x,y) \\
\dot{y} &= g_0(x,y) + h g_1(x,y) + \ldots + h^k g_k(x,y), 
\end{split}
\end{equation}
be the $k$-th truncation of the principal modified equation for the difference equation described by the discrete Lagrangian $\widehat{L}$ from Proposition \ref{prop-extended}. Then \eqref{fullmod} is the $k$-th truncation of the full system of modified equations for the variational integrator described by $L$.
\end{theorem}
\begin{proof}
Let $(x(t),y(t))$ be a solution of the system \eqref{fullmod}. By definition of the principal modified equation, the discrete curve
\[ \big( x(t+jh),y(t+jh) \big)_\jj \]
satisfies the discrete Euler-Lagrange equations for $\widehat{L}$ up to a truncation error for any choice of $t$. Hence, by Proposition \ref{prop-extended}, the discrete curve
\[ \big( x(t+jh)+(-1)^j y(t+jh) \big)_\jj \]
satisfies the discrete Euler-Lagrange equations for $L$ up to a truncation error. This is the defining property of the system of modified equations, see Definition \ref{defi}$(b)$. 
\end{proof}

\begin{cor}
Up to a truncation error of arbitrarily high order, the full system of modified equations \eqref{fullmod} for a variational integrator is Lagrangian.
\end{cor}

Let us illustrate this construction by applying it to our two methods.

\subsubsection*{Midpoint rule}

We have
\begin{align*}
\widehat{L}_\disc(x_j,y_j,x_{j+1},y_{j+1},h) &= \frac{1}{2} \inner{ \alpha \!\left( \frac{x_j+y_j+x_{j+1}-y_{j+1}}{2} \!\right) }{ \frac{x_{j+1} - y_{j+1} - x_j - y_j}{h} } \\
&\quad + \frac{1}{2} \inner{  \alpha \!\left( \frac{x_j-y_j+x_{j+1}+y_{j+1}}{2} \right)\! }{ \frac{x_{j+1} + y_{j+1} - x_j + y_j}{h} } \\
& - \frac{1}{2} H \!\left( \frac{x_j+y_j+x_{j+1}-y_{j+1}}{2} \right) - \frac{1}{2} H \!\left( \frac{x_j-y_j+x_{j+1}+y_{j+1}}{2} \right).
\end{align*}
Hence
\begin{align*}
\widehat{\cL}_\disc([x,y],h) &= \frac{1}{2} \inner{ \alpha \!\left( x - \frac{h}{2}\dot{y} \right)\! }{ \dot{x} - \frac{2}{h} y } + \frac{1}{2} \inner{ \alpha \!\left( x + \frac{h}{2}\dot{y} \right)\! }{ \dot{x} + \frac{2}{h} y } -  H(x) + \O(h) \\
&= \inner{ \alpha(x) }{ \dot{x}} + \inner{ \alpha'(x) \dot{y} }{ y } -  H(x) + \O(h) .
\end{align*}
This is also the leading order term of the modified Lagrangian, $\widehat{\cL}_{\mod,0}(x,y,\dot{x},\dot{y},h)$. If $\alpha$ is linear, its Euler-Lagrange equations are
\begin{align*}
\dot{x} &= A_\sk^{-1} H'(x)^T + \O(h), \\
\dot{y} &= 0 + \O(h).
\end{align*}
Since $y$ is constant in leading order, we need to look at higher order terms to determine whether parasitic solutions occur. No higher order terms of the modified Lagrangian contain $y$ itself, and those terms that contain derivatives of $y$ are at least quadratic in the derivatives of $y$. From these observations one can deduce that the parasitic modified equation is $\dot{y} = 0$ to any order of accuracy. It follows that the parasitic oscillations are of constant magnitude. Hence if the initialization of the discrete system is close to a solution of the principal modified equation, then the discrete solution will remain close to it.

\subsubsection*{Trapezoidal rule}

We have
\begin{align*}
\widehat{L}_\disc(x_j,y_j,x_{j+1},y_{j+1},h) &= \frac{1}{4} \inner{ \alpha(x_j+y_j) + \alpha(x_{j+1}-y_{j+1}) }{ \frac{x_{j+1} - y_{j+1} - x_j - y_j}{h} } \\
&\quad + \frac{1}{4} \inner{  \alpha(x_j-y_j) + \alpha(x_{j+1}+y_{j+1}) }{ \frac{x_{j+1} + y_{j+1} - x_j + y_j}{h} } \\
&\quad - \frac{1}{4} H(x_j+y_j) - \frac{1}{4} H(x_{j+1}-y_{j+1}) \\
&\quad - \frac{1}{4} H(x_j-y_j) - \frac{1}{4} H(x_{j+1}+y_{j+1}).
\end{align*}
Hence
\begin{align*}
\widehat{\cL}_\disc & ([x,y],h) 
= \frac{1}{4} \inner{ \alpha \!\left(x+y-\frac{h}{2}\dot{x}-\frac{h}{2}\dot{y} \right)\! + \alpha \!\left(x-y+\frac{h}{2}\dot{x}-\frac{h}{2}\dot{y} \right)\! }{ \dot{x} - \frac{2}{h} y } \\
&\qquad + \frac{1}{4} \inner{ \alpha \!\left(x-y-\frac{h}{2}\dot{x}+\frac{h}{2}\dot{y} \right)\! + \alpha \!\left(x+y+\frac{h}{2}\dot{x}+\frac{h}{2}\dot{y} \right)\! }{ \dot{x} + \frac{2}{h} y } \\
&\qquad - \frac{1}{2} H(x+y) - \frac{1}{2} H(x-y) + \O(h) 
\\
&= \frac{1}{4} \inner{ \alpha(x \!+\! y)-\frac{h}{2}\alpha'( x \!+\! y)(\dot{x} + \dot{y}) + \alpha(x \!-\! y)+\frac{h}{2}\alpha'(x \!-\! y)(\dot{x} - \dot{y}) }{ \dot{x} - \frac{2}{h} y } \\
&\qquad + \frac{1}{4} \inner{ \alpha(x \!-\! y)-\frac{h}{2}\alpha'(x \!-\! y)(\dot{x}-\dot{y}) + \alpha(x \!+\! y)+\frac{h}{2}\alpha'(x \!+\! y)(\dot{x}+\dot{y}) }{ \dot{x} + \frac{2}{h} y } \\
&\qquad - \frac{1}{2} H(x+y) - \frac{1}{2} H(x-y) + \O(h) 
\\
&= \frac{1}{2} \inner{\alpha(x+y)}{\dot{x}} + \frac{1}{2} \inner{\alpha(x-y)}{\dot{x}} + \frac{1}{2} \inner{\alpha'(x+y)(\dot{x}+\dot{y})}{y} \\
&\qquad - \frac{1}{2} \inner{\alpha'(x-y)(\dot{x}-\dot{y})}{y} - \frac{1}{2} H(x+y) - \frac{1}{2} H(x-y) + \O(h).
\end{align*}
This is also the leading order term of the modified Lagrangian, $\widehat{\cL}_{\mod,0}(x,y,\dot{x},\dot{y},h)$. If $\alpha$ is linear, $\alpha(q)=Aq$, then we find
\[ \widehat{\cL}_{\mod,0}(x,y,\dot{x},\dot{y},h) = \inner{Ax}{\dot{x}} + \inner{A \dot{y}}{y} - \frac{1}{2} H(x+y) - \frac{1}{2} H(x-y). \]
Its Euler-Lagrange equations are
\begin{align*}
\dot{x} &= A_\sk^{-1} \left( \frac{1}{2} H'(x+y)^T + \frac{1}{2} H'(x-y)^T \right) + \O(h), \\
\dot{y} &= A_\sk^{-1} \left( -\frac{1}{2} H'(x+y)^T + \frac{1}{2} H'(x-y)^T \right) + \O(h).
\end{align*}
We linearize the second equation around $y=0$ and find
\begin{equation}\label{parasites-lin}
\dot{y} = - A_\sk^{-1} H''(x) y + \O(|y|^2 + h).
\end{equation}
Heuristically we would expect exponentially growing parasitic solutions if the matrix $-A_\sk^{-1} H''(x)$ has at least one eigenvalue with positive real part. However, since this matrix is not constant it is difficult to give a general condition for the occurrence of exponentially growing parasites. This has to be investigated on a case-by-case basis.

\subsection{The direct approach}

If one is not interested in the Lagrangian structure of the problem, it might be preferable to use a more direct approach to calculate the modified equation. We demonstrate this method in the case of linear $\alpha$ for our two integrators. For more details we refer to \cite{hairer1999backward}.

\subsubsection*{Midpoint rule}

In the difference equation
\[ \frac{q_{j+1} - q_{j-1}}{2h} = A_\sk^{-1}\left(\frac{1}{2} H'\!\left(\frac{q_{j-1} + q_{j}}{2}\right)^T + \frac{1}{2} H'\!\left(\frac{q_{j} + q_{j+1}}{2}\right)^T \right) \]
we set $q_j = x(t) + (-1)^j y(t)$ and
\begin{align*}
q_{j \pm 1} 
&= x(t \pm h) + (-1)^{j \pm 1} y(t \pm h) \\
&= \left( x(t) \pm h \dot{x}(t) + \frac{h^2}{2} \ddot{x}(t) \pm  \ldots \right)  - (-1)^j \left( y(t) \pm h \dot{y}(t) + \frac{h^2}{2} \ddot{y}(t) \pm \ldots \right).
\end{align*}
It follows that
\[
\frac{q_{j+1} - q_{j-1}}{2h} = \dot{x}(t) - (-1)^j \dot{y}(t) + \O(h^2)
\]
and
\begin{align*}
H' \!\left( \frac{q_j + q_{j \pm 1}}{2} \right) 
&= H' \!\left(x \pm \frac{h}{2}\dot{x} \pm \frac{h}{2} (-1)^{j+1} \dot{y} \right) + \O(h^2)  \\
&= H'(x) \pm \frac{h}{2} H''(x) \left( \dot{x} + (-1)^{j+1} \dot{y} \right) + \O(h^2) .
\end{align*}
Hence
\begin{align*}
\dot{x} - (-1)^j \dot{y} &= A_\sk^{-1} \bigg( H'(x) +  \frac{h}{4} H''(x) \left( \dot{x} + (-1)^{j+1} \dot{y} \right) -  \frac{h}{4} H''(x) \left( \dot{x} + (-1)^{j+1} \dot{y} \right) \bigg)^T \\
&\hspace{11cm}  + \O(h^2) \\
&= A_\sk^{-1} H'(x)^T  + \O(h^2).
\end{align*}
Separating the alternating terms from the rest, we find
\begin{align*}
\dot{x} &= A_\sk^{-1} H'(x)^T + \O(h^2) , \\
\dot{y} &= 0+ \O(h^2).
\end{align*}
Unsurprisingly, we find the same system of modified equations as with the Lagrangian method.

\subsubsection*{Trapezoidal rule}

Now we consider the difference equation
\[ \frac{q_{j+1} - q_{j-1}}{2h} = A_\sk^{-1}H'(q_j)^T \]
and make the same identifications as before. We find
\begin{align*}
\dot{x} - (-1)^j \dot{y} &= A_\sk^{-1} H'(x + (-1)^j y)^T  + \O(h^2)\\
&= A_\sk^{-1} H'(x)^T + (-1)^j A_\sk^{-1} H''(x) y + \O(y^2 + h^2).
\end{align*}
If we assume that $y = \O(h)$, then the system of modified equations is
\begin{align*}
\dot{x} &= A_\sk^{-1} H'(x)^T + \O(h^2) , \\
\dot{y} &= -A_\sk^{-1} H''(x) y + \O(h^2).
\end{align*}

\section{Examples}\label{sect-ex}

To illustrate the theory above, we apply our two integrators to two examples. Since the calculations tend to be quite long in real-world problems, we start with a minimal toy problem. After that, we discuss the dynamics of point vortices in the plane.

\subsection{Toy Problem}

Consider the Lagrangian 
\[ \cL(p,q,\dot{p},\dot{q}) = \frac{1}{2}(p\dot{q}-q\dot{p})-U(p)-V(q) \]
on $T\mathbb{R}^2$. Its Euler-Lagrange equations are
\[ \dot{p} = -V'(q) \qquad \text{and} \qquad \dot{q} = U'(p). \]
As a concrete example, the choice $V(q)= - \cos(q)$ and $U(p)=\frac{1}{2}p^2$ describes the pendulum.

\subsubsection*{Midpoint rule}

We have
\begin{align*}
L_\disc(p_j,q_j,p_{j+1},q_{j+1},h) &= \frac{1}{2} \left( \frac{p_j+p_{j+1}}{2}\frac{q_{j+1}-q_j}{h} - \frac{q_j+q_{j+1}}{2}\frac{p_{j+1}-p_j}{h} \right) \\
 &\quad - U\!\left( \frac{p_j + p_{j+1}}{2} \right) - V\!\left( \frac{q_j + q_{j+1}}{2} \right) .
\end{align*}
This corresponds to the following system of difference equations:
\begin{align*}
\frac{q_{j+1} - q_{j-1}}{2h} &= \frac{1}{2} U' \!\left( \frac{p_{j-1} + p_j}{2} \right) + \frac{1}{2} U' \!\left( \frac{p_j + p_{j+1}}{2} \right) , \\
\frac{p_{j+1} - p_{j-1}}{2h} &= -\frac{1}{2} V' \!\left( \frac{q_{j-1} + q_j}{2} \right) - \frac{1}{2} V' \!\left( \frac{q_j + q_{j+1}}{2} \right) .
\end{align*}
By Taylor expansion we obtain
\begin{align*}
\cL_\disc([p,q],h) &= \cL(p,q,\dot{p},\dot{q}) \\
&\quad + \frac{h^2}{24} \left( \frac{1}{2} \left( pq^{(3)} + 3 \ddot{p}\dot{q} - 3 \dot{p}\ddot{q} - p^{(3)}q \right) - 3 U'\ddot{p} - 3 V'\ddot{q} \right) + \O(h^4).
\end{align*}
It follows that
\[ \cL_\mesh([p,q],h) = \cL(p,q,\dot{p},\dot{q}) + \frac{h^2}{24} \left( 2\ddot{p}\dot{q} - 2\dot{p}\ddot{q} - 2U'\ddot{p} + U''\dot{p}^2 - 2V'\ddot{q} + V''\dot{q}^2 \right) + \O(h^4) . \]
Its Euler-Lagrange equations are
\begin{align*}
0 &= \der{\cL_\mesh}{p} - \frac{\d}{\d t}\der{\cL_\mesh}{\dot{p}} 
= \dot{q} - U' + \frac{h^2}{24} \left( 2 q^{(3)} - U^{(3)}\dot{p}^2 - 4 U''\ddot{p} \right) + \O(h^4), \\
0 &= \der{\cL_\mesh}{q} - \frac{\d}{\d t}\der{\cL_\mesh}{\dot{q}} 
= -\dot{p} - V' + \frac{h^2}{24} \left( -2 p^{(3)} - V^{(3)}\dot{q}^2 - 4 V''\ddot{q} \right) + \O(h^4).
\end{align*}
Solving for $q$ and $\dot{q}$ we find the principal modified equations
\begin{align*}
\dot{q} &= U' - \frac{h^2}{24} \left( U^{(3)}V'^2 + 2 U''V''U' \right) + \O(h^4) ,\\
\dot{p} &= -V' + \frac{h^2}{24} \left( V^{(3)}U'^2 + 2 V''U''V' \right) + \O(h^4) .
\end{align*}
Eliminating higher derivatives in $\cL_\mesh$ we find
\[ \cL_\mod(p,q,\dot{p},\dot{q},h) = \cL(p,q,\dot{p},\dot{q}) + \frac{h^2}{24} \left( - V'' \dot{q}^2 - U'' \dot{p}^2 + 2 U'V''\dot{q} - 2 V'U''\dot{p} \right) + \O(h^4) . \]
As discussed in the previous section we do not expect parasitic solutions with this method (see Figure \ref{fig-pendulum}). 

\subsubsection*{Trapezoidal rule}

We have
\begin{align*}
L_\disc(p_j,q_j,p_{j+1},q_{j+1},h) &= \frac{1}{2} \left( \frac{p_j+p_{j+1}}{2}\frac{q_{j+1}-q_j}{h} - \frac{q_j+q_{j+1}}{2}\frac{p_{j+1}-p_j}{h} \right) \\
 &\quad - \frac{1}{2} U(p_j) - \frac{1}{2} U(p_{j+1}) - \frac{1}{2} V(q_j) - \frac{1}{2} V(q_{j+1}) .
\end{align*}
The corresponding discrete Euler-Lagrange equations are
\begin{align*}
\frac{q_{j+1} - q_{j-1}}{2h} = U'(p_j) , \qquad
\frac{p_{j+1} - p_{j-1}}{2h} = -V'(q_j) .
\end{align*}
By Taylor expansion we obtain
\begin{align*}
\cL_\disc([p,q],h) &= \cL(p,q,\dot{p},\dot{q}) \\
&\quad + \frac{h^2}{24} \left( \frac{1}{2} \left( pq^{(3)} + 3 \ddot{p}\dot{q} - 3 \dot{p}\ddot{q} - p^{(3)}q \right) - 3 U'\ddot{p} - 3 U''\dot{p}^2 - 3 V'\ddot{q} - 3 V''\dot{q}^2 \right) \\
&\hspace{10.8cm} + \O(h^4).
\end{align*}
It follows that
\[ \cL_\mesh([p,q],h) = \cL(p,q,\dot{p},\dot{q}) + \frac{h^2}{12} \left( \ddot{p}\dot{q} - \dot{p}\ddot{q} - U'\ddot{p} - U''\dot{p}^2 - V'\ddot{q} - V''\dot{q}^2 \right) + \O(h^4) . \]
Its Euler-Lagrange equations are
\begin{align*}
0 &= \der{\cL_\mesh}{p} - \frac{\d}{\d t}\der{\cL_\mesh}{\dot{p}} 
= \dot{q} - U' + \frac{h^2}{12} \left( q^{(3)} + U^{(3)}\dot{p}^2 + U''\ddot{p} \right) + \O(h^4), \\
0 &= \der{\cL_\mesh}{q} - \frac{\d}{\d t}\der{\cL_\mesh}{\dot{q}} 
= -\dot{p} - V' + \frac{h^2}{12} \left( -p^{(3)} + V^{(3)}\dot{q}^2 + V''\ddot{q} \right) + \O(h^4).
\end{align*}
Solving for $q$ and $\dot{q}$ we find the principal modified equations
\begin{align*}
\dot{q} &= U' - \frac{h^2}{6} \left( U^{(3)}V'^2 - U''V''U' \right) + \O(h^4) ,\\
\dot{p} &= -V' + \frac{h^2}{6} \left( V^{(3)}U'^2 - V''U''V' \right) + \O(h^4) .
\end{align*}
Eliminating higher derivatives in $\cL_\mesh$ we find
\[ \cL_\mod(p,q,\dot{p},\dot{q},h) = \cL(p,q,\dot{p},\dot{q}) + \frac{h^2}{12} \left( -2 V'' \dot{q}^2 -2 U'' \dot{p}^2 + U'V''\dot{q} - V'U''\dot{p} \right) + \O(h^4) . \]

For the pendulum, $V(q)= - \cos(q)$ and $U(p)=\frac{1}{2}p^2$, we have
\[
A = \frac{1}{2}
\begin{pmatrix}
0 & 1 \\
-1 & 0
\end{pmatrix}
\qquad \text{and} \qquad 
H'' = 
\begin{pmatrix}
U''(p) & 0 \\
0 & V''(q)
\end{pmatrix}
 = 
\begin{pmatrix}
1 & 0 \\
0 & \cos(q)
\end{pmatrix}, \]
hence the matrix in Equation \eqref{parasites-lin} is
\[ -A_\sk^{-1} H'' = \begin{pmatrix}
0 & -\cos(q) \\
1 & 0
\end{pmatrix}.\]
This matrix has a pair of real eigenvalues if $\cos(q) < 0$ and a pair of purely imaginary eigenvalues if $\cos(q) > 0$. This suggests (but does not prove; $q$ is not constant) that exponentially growing parasites occur in the regions where $\cos(q) < 0$. 

In the top right image of Figure \ref{fig-pendulum} one clearly observes parasitic solutions for this method. Note the parasites only seem to grow where $|q| > \frac{\pi}{2}$, i.e.\@ where $\cos(q) < 0$. In the region where $|q| < \frac{\pi}{2}$ there is no noticeable growth in the amplitude of the oscillations. Instead we observe a rotation in the direction of the oscillations, as expected when the eigenvalues are purely imaginary. This is visualized in Figure \ref{fig-pendulum} by line segments connecting the points of the discrete solution with the corresponding points on the solution of the principal modified equation.

When the initial conditions are chosen such that $q$ remains in the stable region $|q| < \frac{\pi}{2}$ no parasites are observed (bottom right image of Figure \ref{fig-pendulum}), even if the simulation is continued for many periods (not pictured).

\begin{figure}[ht]
\includegraphics[width=.47\linewidth]{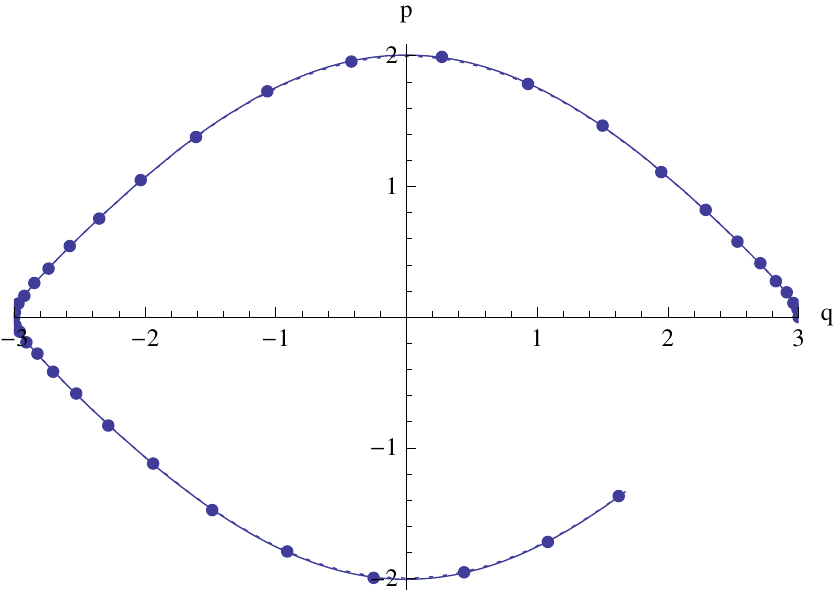}
\hfill
\includegraphics[width=.47\linewidth]{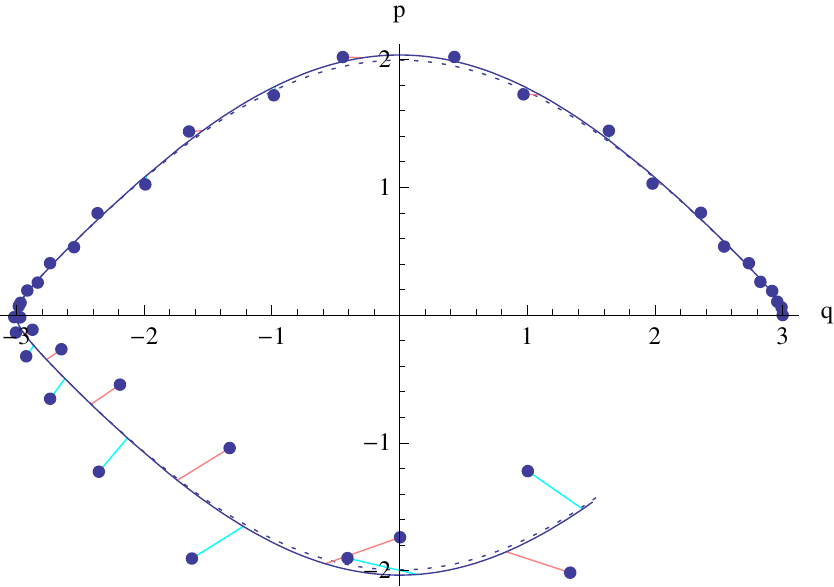}
\bigskip

\includegraphics[width=.47\linewidth]{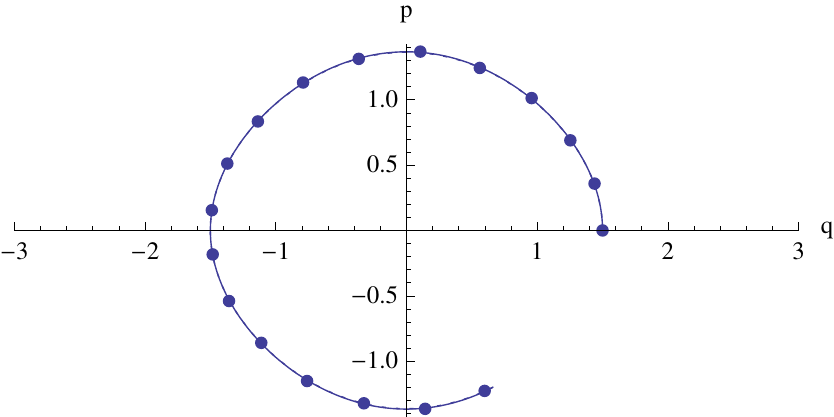}
\hfill
\includegraphics[width=.47\linewidth]{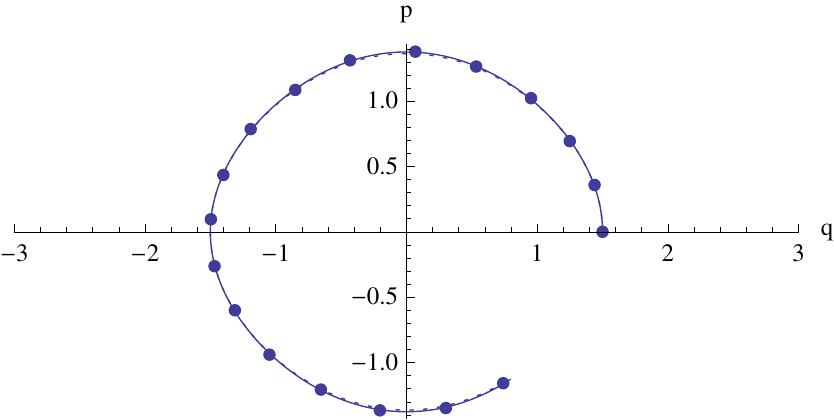}

\caption{Pendulum with midpoint rule (left) and trapezoidal rule (right), both with step size $h=0.35$ and initial point $(3,0)$ (top) and $(1.5,0)$ (bottom).\\[1mm]
Dashed curve: exact solution. \\[1mm]
Bullets: discrete solution. \\[1mm]
Solid curve: solution of the second truncation of the principal modified equation.\\[1mm]
Line segments: visualization of parasitic oscillations}
\label{fig-pendulum}
\end{figure}

\subsection{Point vortices}
\label{sect-vortex}

Our second example involves vortices on a planar surface. If all vorticity is contained in a finite number of points, then the movement of those points is described by first order ODEs \cite{newton2013n,rowley2002variational}. To be precise, the dynamics of $N$ point vortices in the (complex) plane is described by the Lagrangian
\[ \cL(z,\ol{z},\dot{z},\dot{\ol{z}}) = \sum_{j=1}^N \Gamma_j \im( \ol{z}_j \dot{z}_j ) - \frac{1}{\pi} \sum_{j=1}^N \sum_{k=1}^{j-1} \Gamma_j \Gamma_k \log \big| z_j-z_k \big|, \]
where $z_j$ and $\Gamma_j$ are the position and circulation of the $j$-th vortex, and the bar denotes the complex conjugate. The equations of motion are
\[
\dot{z}_j = \frac{i}{2 \pi} \sum_{k \neq j} \frac{\Gamma_k}{\ol{z}_j - \ol{z}_k} \qquad \text{for } j = 1, \ldots, N.
\]
It follows that
\begin{equation}\label{diffvortex}
\ddot{z}_j 
= \frac{i}{2 \pi} \sum_{k \neq j} \frac{-\Gamma_k}{\left( \ol{z}_j - \ol{z}_k \right)^2} \left( \dot{\ol{z}}_j - \dot{\ol{z}}_k \right)  .
\end{equation}

\subsubsection*{Midpoint rule}

We have
\begin{align*}
&\cL_\disc([z,\ol{z}],h) \\
&= \cL(z,\ol{z},\dot{z},\dot{\ol{z}})
+ \frac{h^2}{24} \left[ \sum_{j=1}^N \Gamma_j \im\!\left( 3 \dot{z}_j \ddot{\ol{z}}_j + z^{(3)}_j \ol{z}_j \right) 
- \sum_{j=1}^N \sum_{k=1}^{j-1} \frac{3 \Gamma_j \Gamma_k}{\pi} \re \!\left( \frac{\ddot{z}_j - \ddot{z}_k}{z_j-z_k} \right) \right] + \O(h^4)
\end{align*}
and
\begin{align*}
&\cL_\mesh([z,\ol{z}],h) \\
&= \cL(z,\ol{z},\dot{z},\dot{\ol{z}}) 
+ \frac{h^2}{24} \left[ 
4 \sum_{j=1}^N \Gamma_j \im \!\left( \dot{z}_j \ddot{\ol{z}}_j \right)\!
- \sum_{j=1}^N \sum_{k=1}^{j-1} \frac{\Gamma_j \Gamma_k}{\pi} \re \!\left( 2\frac{\ddot{z}_j - \ddot{z}_k}{z_j-z_k} + \left( \frac{\dot{z}_j-\dot{z}_k}{z_j-z_k} \right)^2 \right)\! \right] \\
&\hspace{13cm} + \O(h^4).
\end{align*}
To obtain the modified Lagrangian we evaluate the second derivatives in $\cL_\mesh$ using the leading order equation \eqref{diffvortex}. We find
\begin{align*}
\sum_{j=1}^N \Gamma_j \im\!\left( \dot{z}_j \ddot{\ol{z}}_j \right) 
&= \sum_{j=1}^N \sum_{k \neq j} \Gamma_j \im \!\left( \dot{z}_j \frac{i}{2 \pi} \frac{\Gamma_k}{(z_j-z_k)^2} \big(\dot{z}_j-\dot{z}_k\big) \right) + \O(h^2) \\  
&= \sum_{j=1}^N \sum_{k \neq j} \frac{\Gamma_j \Gamma_k}{2\pi} \re \!\left( \dot{z}_j  \frac{\dot{z}_j-\dot{z}_k}{(z_j-z_k)^2} \right) + \O(h^2)  \\   
&= \sum_{j=1}^N \sum_{k \neq j} \frac{\Gamma_j \Gamma_k}{4\pi} \re\!\left(   \frac{\big(\dot{z}_j-\dot{z}_k\big)^2}{(z_j-z_k)^2} \right) + \O(h^2) 
\end{align*}
and
\begin{align*}
\sum_{j=1}^N \sum_{k \neq j} \Gamma_j \Gamma_k \re \!\left( 2\frac{\ddot{z}_j - \ddot{z}_k}{z_j-z_k} \right)
&= 4 \sum_{j=1}^N \sum_{k \neq j} \Gamma_j \Gamma_k \re\!\left( \frac{\ddot{z}_j}{z_j-z_k} \right) + \O(h^2)  \\
&= 4 \sum_{j=1}^N \sum_{k \neq j}\sum_{\ell \neq j} \Gamma_j \Gamma_k  \re\!\left( \frac{-i}{2\pi} \frac{\Gamma_\ell \big(\dot{\ol{z}}_j-\dot{\ol{z}}_\ell\big)}{(z_j-z_k)\big(\ol{z}_j-\ol{z}_\ell\big)^2} \right) + \O(h^2) \\
&= \frac{1}{\pi} \sum_{j=1}^N \sum_{k \neq j}\sum_{\ell \neq j} \Gamma_j \Gamma_k \Gamma_\ell \im\!\left( \frac{ \big(\dot{\ol{z}}_j-\dot{\ol{z}}_\ell\big)}{(z_j-z_k)\big(\ol{z}_j-\ol{z}_\ell\big)^2} \right) + \O(h^2) .
\end{align*}
Therefore,
\begin{align*}
\cL_\mod(z,\ol{z},\dot{z},\dot{\ol{z}},h) &= \cL(z,\ol{z},\dot{z},\dot{\ol{z}})
 + \frac{h^2}{24} \left[ 
\frac{1}{2\pi} \sum_{j=1}^N \sum_{k \neq j} \Gamma_j \Gamma_k \re\!\left(  \left( \frac{\dot{z}_j-\dot{z}_k}{z_j-z_k} \right)^2 \right) \right. \\
&\hspace{1cm}\left.
- \frac{1}{\pi^2} \sum_{j=1}^N \sum_{k \neq j} \sum_{\ell \neq j} \Gamma_j \Gamma_k \Gamma_\ell \im\!\left( \frac{ \big(\dot{\ol{z}}_j-\dot{\ol{z}}_\ell\big)}{(z_j-z_k)\big(\ol{z}_j-\ol{z}_\ell\big)^2} \right) \right]
 + \O(h^4).
\end{align*}

\begin{figure}[p]
\includegraphics[width=\linewidth]{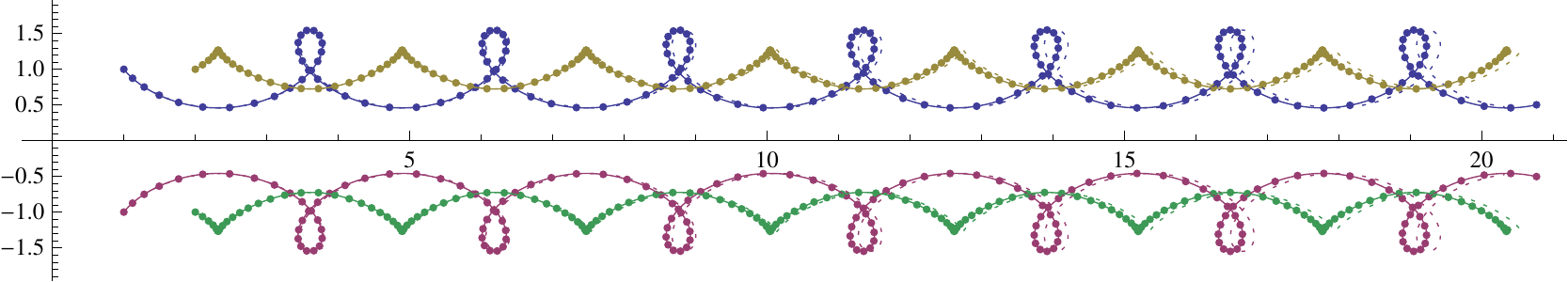}
\caption{Leapfrogging vortex pairs with the midpoint rule. No parasitic behavior is visible.}
\label{fig-vortex-mp}
\end{figure}
\begin{figure}[p]
\includegraphics[width=\linewidth]{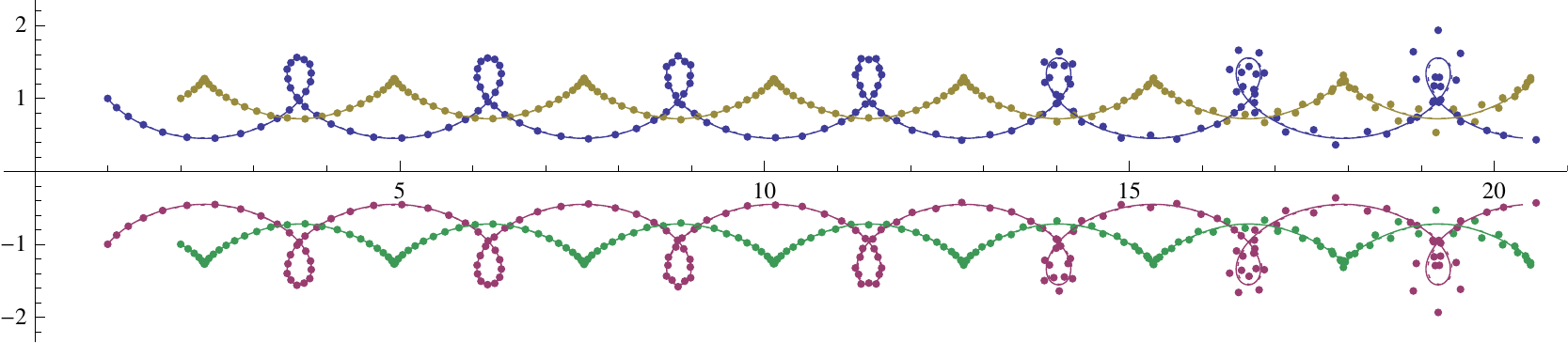}
\caption{Leapfrogging vortex pairs with the trapezoidal rule. One observes parasitic oscillations.}
\label{fig-vortex-sv}
\end{figure}
\begin{figure}[p]
\includegraphics[width=.45\linewidth]{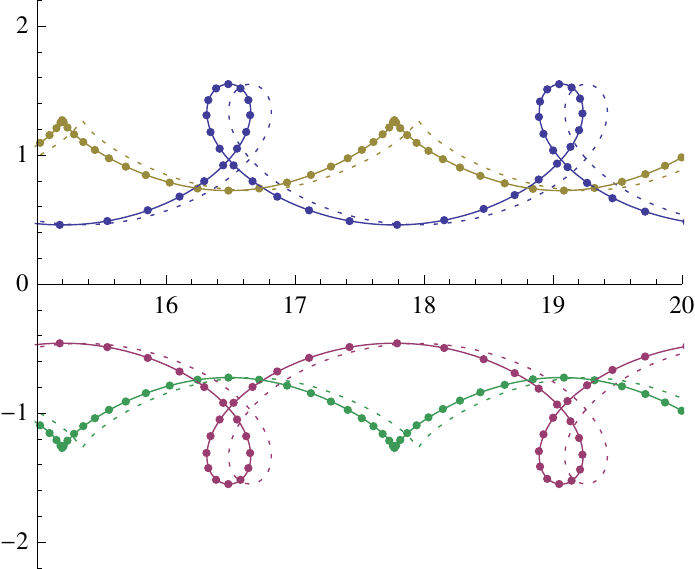}
\hfill
\includegraphics[width=.45\linewidth]{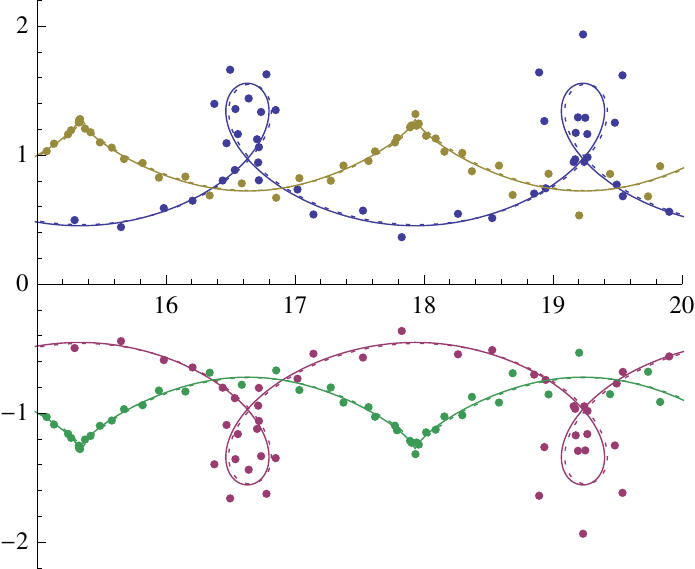}
\caption{Enlarged versions of the right hand sections of Figures \ref{fig-vortex-mp} and \ref{fig-vortex-sv}: midpoint rule (left) and trapezoidal rule (right).}
\label{fig-vortex-zoom}
\end{figure}
\begin{figure}[p]%
\hspace{-2mm}%
\begin{tabular}{lrl}
\textbf{Legend:} & Dashed curves & \emph{exact solution.}\\
& Bullets & \emph{discrete solution.} \\
& Solid curves & \emph{solution of the truncated principal modified equation.}\\
\\
\textbf{Parameters:} & Initial positions & \emph{$(1,1)$, $(1,-1)$, $(2,1)$, and $(2,-1)$.}\\
& Vortex strengths & \emph{$1$, $-1$, $2$, $-2$, respectively.}\\
& Time interval & \emph{$0 \leq t \leq 80$.}
\end{tabular}
\end{figure}

\subsubsection*{Trapezoidal rule}

For the Trapezoidal rule, we find in the same way that
\begin{align*}
& \cL_\mesh([z,\ol{z}],h) \\
&= \cL(z,\ol{z},\dot{z},\dot{\ol{z}}) 
+ \frac{h^2}{24} \left[ 
4 \sum_{j=1}^N \Gamma_j \im\!\left( \dot{z}_j \ddot{\ol{z}}_j \right) 
- 2 \sum_{j=1}^N \sum_{k=1}^{j-1} \frac{\Gamma_j \Gamma_k}{\pi} \re\!\left( \frac{\ddot{z}_j - \ddot{z}_k}{z_j-z_k} - \left( \frac{\dot{z}_j-\dot{z}_k}{z_j-z_k} \right)^2 \right) \right] \\
&\hspace{13cm} + \O(h^4)
\end{align*}
and
\begin{align*}
\cL_\mod(z,\ol{z},\dot{z},\dot{\ol{z}},h) &= \cL(z,\ol{z},\dot{z},\dot{\ol{z}})
 + \frac{h^2}{24} \left[ 
\frac{2}{\pi} \sum_{j=1}^N \sum_{k \neq j} \Gamma_j \Gamma_k \re\!\left(  \left( \frac{\dot{z}_j-\dot{z}_k}{z_j-z_k} \right)^2 \right) \right. \\
&\hspace{1cm}\left.
- \frac{1}{\pi^2} \sum_{j=1}^N \sum_{k \neq j} \sum_{\ell \neq j} \Gamma_j \Gamma_k \Gamma_\ell \im\!\left( \frac{ \big(\dot{\ol{z}}_j-\dot{\ol{z}}_\ell\big)}{(z_j-z_k)\big(\ol{z}_j-\ol{z}_\ell\big)^2} \right) \right]
 + \O(h^4).
\end{align*}

In Figures \ref{fig-vortex-mp}\---\ref{fig-vortex-zoom} we observe parasitic solutions for the trapezoidal rule, but not for the midpoint rule, where the solution of the principal modified equation shows excellent agreement with the discrete solution.

\section{Conclusion}

We have described a Lagrangian algorithm to calculate the modified equation of a variational integrator applied to a degenerate continuous Lagrangian that is linear in the velocities. To obtain the principal modified equation this was a straightforward adaptation of the procedure developed for non-degenerate Lagrangians. To obtain the full system of modified equations we doubled the dimension of the discrete system in a suitable way. As a consequence, we proved that the system of modified equations is variational. We have illustrated the construction of modified Lagrangians and the possible issue of parasitic solutions with examples. Our construction is potentially useful to create more accurate variational integrators, for example in the spirit of \citep{torre2016benefits,torre2017surrogate}.

\bigskip\noindent\textbf{Acknowledgment.} The author is funded by the DFG Collaborative Research Center SFB/TRR 109 ``Discretization in Geometry and Dynamics''.

\bibliographystyle{abbrv}
\bibliography{deglag}

\end{document}